\documentclass{amsart}
\usepackage{amsmath,amssymb}
\usepackage{graphicx}
\newtheorem{theorem}{Theorem}[section]
\newtheorem{proposition}[theorem]{Proposition}

\theoremstyle{definition}

\theoremstyle{remark}
\newtheorem{remark}[theorem]{Remark}

\numberwithin{equation}{section}

\newcommand{\al}{\alpha}
\newcommand{\be}{\beta}

\newcommand{\la}{\lambda}

\newcommand{\te}{\theta}
\newcommand{\vp}{\varphi}

\newcommand{\De}{\Delta}
\newcommand{\Ga}{\Gamma}

\newcommand{\Si}{\Sigma}

\def\RR{\mathbb{R}}

\renewcommand\SS{\mathbb{S}}
\newcommand{\cT}{{\mathcal T}}

\newcommand{\pd}{\partial}
\newcommand\minus\backslash

\newcommand\lan\langle
\newcommand\ran\rangle

\DeclareMathOperator\Div{div}

\renewcommand\leq\leqslant
\renewcommand\geq\geqslant
\newlength{\intwidth}

 \DeclareMathOperator\curl{curl}

\begin{document}

\title[Beltrami fields with a nonconstant factor are rare]{Beltrami fields with a nonconstant proportionality factor are rare}

\author{Alberto Enciso}
\address{Instituto de Ciencias Matem\'aticas, Consejo Superior de
  Investigaciones Cient\'\i ficas, 28049 Madrid, Spain}
\email{aenciso@icmat.es}

\author{Daniel Peralta-Salas}
\address{Instituto de Ciencias Matem\'aticas, Consejo Superior de
  Investigaciones Cient\'\i ficas, 28049 Madrid, Spain}
\email{dperalta@icmat.es}

%
%
\begin{abstract}
We consider the existence of Beltrami fields with a nonconstant proportionality factor $f$
in an open subset $U$ of $\RR^3$. By reformulating this problem as a
constrained evolution equation on a surface, we find an explicit
differential equation that $f$ must satisfy whenever there is a
nontrivial Beltrami field with this factor. This ensures that there
are  no nontrivial regular solutions for an open
and dense set of factors $f$ in the $C^k$~topology. In particular,
there are no nontrivial Beltrami fields whenever $f$ has a regular
level set diffeomorphic to the sphere. This provides an explanation
of the helical flow paradox of Morgulis, Yudovich and Zaslavsky ({\em
  Comm.\ Pure Appl.\ Math.}~48 (1995) 571--582).
\end{abstract}
\maketitle

\section{Introduction}
\label{S:intro}

A {\em Beltrami field}\/ is a vector field $u$ in $\RR^3$ such that
\begin{equation}\label{Beltrami}
\curl u= fu\,,\qquad \Div u=0\,,
\end{equation}
where $f$ is a smooth function. The condition that $u$ be
divergence-free is redundant when the proportionality factor $f$ is a
nonzero constant (i.e., in the case of {\em strong}\/ Beltrami fields), while
otherwise it is tantamount to demanding that the function $f$ be a first integral of
$u$, that is,
\begin{equation}\label{firstint}
u\cdot \nabla f=0\,.
\end{equation}

Beltrami fields have been studied since the XIX~century because of
their connection with the Euler equation and with
magnetohydrodynamics, where they are known as force-free
fields. Indeed, it is well known that a Beltrami field is also a
solution of the steady Euler equation in $\RR^3$,
\[
(u\cdot\nabla)u=-\nabla P\,,\qquad \Div u=0
\]
with $P=-\frac12|u|^2$, and actually the analysis of concrete examples of Beltrami fields
with constant proportionality factor such as the ABC
flows~\cite{AK99} has yielded considerable insight e.g.\ into the
phenomenon of Lagrangian turbulence~\cite{Dombre}.

Beyond the study of explicit examples, Beltrami fields with constant
proportionality factor have found application as powerful tools to
analyze the structure of solutions to the Euler equation. For
instance, de Lellis and Sz\'ekelyhidi have utilized strong Beltrami
fields to construct H\"older continuous weak solutions to the Euler
equation in the $3$-torus that dissipate
energy~\cite{Lellis09,Lellis13}, while in Refs.~\cite{Annals,arxiv} we
constructed strong Beltrami fields in $\RR^3$ having vortex lines and
vortex tubes (that is, integral curves and invariant tori) of
arbitrary topology. Expansions of more general solutions to the Euler
equation in terms of strong Beltrami fields were also considered
in~\cite{Constantin}. 

On the contrary, Beltrami fields with nonconstant proportionality
factor have not found as many applications, and indeed to the best of
our knowledge there are just a handful of explicit examples, all of
which have Euclidean symmetries. In fact, the analysis of Beltrami
fields with nonconstant factor has proved
to be extremely hard, as one can infer from the striking lack of
results in this classical subject. An interesting contribution in this
direction is the construction of low-regularity Beltrami fields with H\"older-continuous
nonconstant factors in~\cite{Bineau,CMP00}.

More precisely, the key question, sometimes called the helical flow
problem, is to ascertain for which functions $f$ there is a nontrivial
vector field satisfying the Eq.~\eqref{Beltrami}. In this regard, a
challenging observation due to Morgulis, Yudovich and
Zaslavsky~\cite{Yudovich} is that one would naively expect ``most'' Beltrami
fields to admit a first integral, since this happens whenever the
function $f$ is nonconstant as a consequence of
Eq.~\eqref{firstint}. These authors refer to this phenomenon as the helical flow paradox.
Physically, this means that the fluid flow defined by a Beltrami field
would generically be laminar, in contrast with the physical
intuition that the fluid should typically present a turbulent
behavior~\cite{Pelz}.

However, since the first integral condition~\eqref{firstint} is
very restrictive, it stands to reason that the Eq.~\eqref{Beltrami}
should not admit any nontrivial solutions for most functions~$f$. Our
objective in this paper is to make this idea precise.

Specifically, our main result asserts that, for a generic function
$f$, the only vector field~$u$ satisfying Eq.~\eqref{Beltrami} is the
trivial one, $u\equiv 0$. This provides an explanation of the helical
flow paradox, as it shows that the hypothetical laminar flow
associated to a nonconstant proportionality factor does not exist
generically. In particular, there are no nontrivial local
Beltrami fields with factor~$f$ unless $f$ belongs to a set of infinite codimension in the
$C^k$~topology. The set of functions for which there can be
nontrivial Beltrami fields is contained in the kernel of certain
complicated nonlinear differential operator that one can 
compute explicitly. Therefore, as a byproduct we obtain an effective necessary condition for
$f$ to admit nontrivial solutions as will be illustrated in
Propositions~\ref{P.example1} and~\ref{P.example2} below.

\begin{theorem}\label{T.main}
  Let $U\subseteq\RR^3$ be a domain and assume that the function $f$
  is nonconstant and of class~$C^{6,\al}$. Suppose that the vector
  field~$u$ satisfies the Eq.~\eqref{Beltrami} in~$U$. Then there is a
  nonlinear partial differential operator $P\neq0$, which can be
  computed explicitly and involves derivatives of order at most~$6$,
  such that $u\equiv 0$ unless $P[f]$ is identically zero in $U$. In
  particular, $u\equiv0$ for all $f$ in an open and dense subset of
  $C^{k}(U)$ with any $k\geq 7$.
\end{theorem}

It should be noticed that Theorem~\ref{T.main} is of purely {\em
  local}\/ nature, as it provides obstructions for the existence of
nontrivial Beltrami fields in any open set and most proportionality
factors. On the other hand, Nadirashvili~\cite{Nad14} has recently
proved a global obstruction in the form of a Liouville theorem for
Eq.~\eqref{Beltrami}, which shows that, for any factor~$f$, there are no Beltrami fields in
the whole space~$\RR^3$ falling off fast enough at infinity.

An easy consequence of the proof of the main result is that if~$f$ has
a regular level set diffeomorphic to the sphere, then the
Eq.~\eqref{Beltrami} does not have any nontrivial solutions. In
particular, there are no Beltrami fields whenever $f$ has local
extrema or is a radial function. This is related to the
classical theorem of Cowling ensuring that
there are no poloidal Beltrami fields with nonconstant
factor and axial symmetry~\cite{Chandra}. Observe that the obstruction
to the existence of solutions with a factor having a spherical level
set does not follow from Arnold's structure theorem because the
vorticity and the velocity are collinear.

\begin{theorem}\label{T.spheres}
Suppose that the function $f$ is
of class $C^{2,\al}$ in a domain $U\subseteq\RR^3$. If a regular level set $f^{-1}(c)$ has a connected component
 in $U$ diffeomorphic to $\SS^2$, then any solution to the
  Eq.~\eqref{Beltrami} in $U$ is identically zero. 
\end{theorem}

Before passing to discuss the proof of these results, a few
comments are in order. Firstly, notice that the reason for which we
have not made any regularity assumptions on $u$ is that it
automatically satisfies the elliptic equation
\[
\De u+\nabla f\times u+f^2 \, u=0\,,
\]
which ensures that $u$ is of class $C^{k+1,\al}$ if $f$ is
$C^{k,\al}$. Furthermore, this shows that $u$ satisfies the unique
continuation property, so $u$ is identically zero in its domain
if it vanishes in any open subset. Secondly, an interesting
consequence of the proof of these results is that the theorems
remain valid if we assume instead that $u$ is a strong Beltrami field,
satisfying
\[
\curl u =\la u
\]
for some nonzero constant $\la$, and $f$ is a first integral in
$U$. Therefore, the first integrals of a strong Beltrami field are
also severely restricted. Thirdly, all the results and proofs remain
valid for the Beltrami equation in an arbitrary Riemannian
3-manifold, but we have restricted ourselves to Euclidean space to
simplify the exposition.

The proof of these theorems, given in Section~\ref{S.proofs}, is based on formulating the Beltrami
equation~\eqref{Beltrami} as a constrained evolution problem. Although
the underlying mathematics are relatively unsophisticated, we regard
this reformulation as the main contribution of the paper. Indeed, one can
show that the Eq.~\eqref{Beltrami} is locally equivalent, in a sense
to be made precise later on, to the assertion that there is a
time-dependent 1-form $\be(t)$ on a surface $\Si$ that satisfies the
evolution equation
\begin{equation}\label{evolution}
\pd_t\be=T(t)\, \be
\end{equation}
together with the constraint
\begin{equation}\label{constraint}
d\be=0\,.
\end{equation}
Here $T(t)$ is a time-dependent tensor field that depends on $f$ and the exterior differential
$d$ is computed with respect to the coordinates on the surface $\Si$,
which, in turn, is a regular level set of $f$. It should be stressed
that this formulation depends strongly on the choice of coordinates; full details are given
in Section~\ref{S.evolution}.

This formulation lays bare the reason for which the Beltrami equation does not
generally admit nonzero solutions: the evolution~\eqref{evolution} is
not generally compatible with the constraint~\eqref{constraint}, and
the resulting compatibility conditions translate into equations that
$f$ and its derivatives must satisfy. In Theorems~\ref{T.main}
and~\ref{T.spheres} we have presented the first two of these
compatibility conditions, but in fact there is a whole hierarchy of
explicitly computable obstructions (with increasingly
cumbersome expressions). Furthermore, it provides an appealing explanation,
without even resorting to the statement of the previous theorems, of
the reason for which the attempts at constructing solutions
to~\eqref{Beltrami} using variational techniques have failed: while
the regularity of the equation is indeed determined by an elliptic
system, its existence is in fact controlled by a constrained
evolution problem for which the existence theory is ill posed.

To conclude, let us emphasize that the key to the obstructions for the
existence of nontrivial solutions to the Eq.~\eqref{Beltrami} is
indeed the requirement that $u$ be divergence-free. In fact, in
Section~\ref{S.remarks} we will show that if this condition is
omitted, there are always solutions in the whole space $\RR^3$
provided that the function $f$ is positive. This case
corresponds to a compressible fluid flow, with $f$ playing the role of
the density of the fluid.

\section{The Beltrami equation as a constrained evolution}
\label{S.evolution}

Our goal in this section is to reformulate the Beltrami
equation~\eqref{Beltrami} as a constrained evolution problem for a
1-form on a surface. The key equations that we derive here
are~\eqref{pdtal} and~\eqref{dal=0}, which were already discussed in
the Introduction.

Let us take a point $p$ of the domain $U$ such that the gradient of
$f$ does not vanish in a small neighborhood (which we still call $U$) of $p$. Without loss
of generality we can assume that 
\[
f(p)=1
\]
and 
\[
\Si:=f^{-1}(1)\cap U
\]
is a connected surface. By rotating the coordinate axes if
necessary, $\Si$ can be parametrized as a graph, namely
\[
\Si=\big\{(\xi,h(\xi))\big\}
\]
with the coordinates $\xi=(\xi_1,\xi_2)$ taking values in a disc. Here $h$
is a function of the same regularity as $f$ and is defined via the implicit function theorem and the relation
\[
f (\xi,h(\xi))=1\,.
\]
Moreover, we can assume that the point $p$
lies at the origin of the coordinate system and that the gradient of $f$ is parallel to the
third coordinate at that point, which means
\begin{gather*}
h(0)=0\quad \text{and} \quad  \frac{\pd h}{\pd\xi_i}(0)=0\,.
\end{gather*}

Let us consider the vector field
\[
X:=\frac{\nabla f}{|\nabla f|^2}
\]
and denote by $\phi_t$ its local flow. We can parametrize $U$
by coordinates $(t,\xi)$ defined via
\begin{equation}\label{xtxi}
x=\phi_t(\xi,h(\xi))\,.
\end{equation}
It is clear that $f=1$ when $t=0$ by the definition of the
function~$h$ and that
\[
X\cdot \nabla f=1\,,
\]
which implies that 
\begin{equation}\label{ft}
f(\phi_sx)=f(x)+s
\end{equation}
as long as the action of the local flow on $x$ is defined.
In particular, we deduce that in the new coordinates the function $f$ reads as
\begin{equation}\label{ft2}
f=1+t\,.
\end{equation}

It is important to notice that, in these coordinates, the Euclidean
metric is of the form
\begin{equation}\label{ortho}
ds^2=\chi(t,\xi)^2\, dt^2+ g_{ij}(t,\xi)\, d\xi_i \, d\xi_j\,,
\end{equation}
where the function $\chi$ stands for the function $1/|\nabla f|$
written in the new coordinates and
\[
g_{ij}:=\pd_i x\cdot \pd_j x
\]
is the induced metric of the surface of constant~$t$.
Here $x$ is given in terms of $(t,\xi)$ by~\eqref{xtxi} and $\pd_i$
henceforth stands for the derivative with respect to $\xi_i$. Since
$|X|=1/|\nabla f|$, the only nontrivial assertion here is that the
crossed terms
\[
\pd_i x\cdot \pd_t x
\]
are zero. The easiest way to see this is to prove that the inverse of
the metric tensor, which we claim to be of the form
\[ 
\left( \begin{array}{ccc}
\chi^2 & 0 & 0 \\
0 & g_{11} & g_{12} \\
0 & g_{21} & g_{22} \end{array} \right)\,,
\]
is indeed read as
\[ 
\left( \begin{array}{ccc}
\chi^{-2} & 0 & 0 \\
0 & g^{11} & g^{12} \\
0 & g^{21} & g^{22} \end{array} \right)\,.
\]
This is immediate, for it is well known that the $(t,i)$ component of
the latter
matrix is precisely
\begin{align*}
\nabla f\cdot \nabla \xi_i &=|\nabla f|^{2} X\cdot \nabla \xi_i\\
&=|\nabla f|^{2}\, \frac d{ds}(\xi_i\circ\phi_s)\\
&=0\,.
\end{align*}
Here we are considering the variables $\xi_i$ as functions of $x$ and
to pass to the last line we have used that, as a consequence
of~\eqref{xtxi} and~\eqref{ft},
\begin{align*}
\xi_i(\phi_sx)=\big(\phi_{1-f(\phi_sx)}\circ\phi_sx \big)_i=\big(\phi_{1-f(x)}x \big)_i=\xi_i(x)
\end{align*}
for all $s$, with the subscript $i$ denoting the $i^{\text{th}}$
component of the point.

Given a solution~$u$ to Eq.~\eqref{Beltrami} in $U$, let us denote by
$\be$ its dual 1-form, computed using the Euclidean metric. The first
integral condition~\eqref{firstint}, together with the block structure of
the metric in these coordinates shown in Eq.~\eqref{ortho}, then imply
that $\be$ must be of the form
\begin{equation}\label{al}
\be=\be_i(t,\xi)\, d\xi_i\,.
\end{equation}
Denoting by $|g|$ the determinant of the matrix $(g_{ij})$, a
straightforward computation then shows that the differential and Hodge
star of $\be$ are as follows:
\begin{align}
d_{\RR^3}\be&= (\pd_1\be_2-\pd_2\be_1)\, d\xi_1\wedge d\xi_2 + \pd_t\be_1\,
dt\wedge d\xi_1 + \pd_t\be_2\,
dt\wedge d\xi_2\,,\label{dal}\\
*_{\RR^3}\be & = \chi|g|^{1/2}\, \big( g^{2i}\be_i\, dt\wedge d\xi_1
-g^{1i}\be_i\, dt\wedge d\xi_2\big)\,.\label{*al}
\end{align}
Here we are using the cumbersome notation $d_{\RR^3}\be$ and $*_{\RR^3}$
to stress that these operations are computed with respect to all
three variables $(t,\xi)$ and thus avoid confusion with the
two-dimensional exterior derivative and Hodge operator that we will introduce shortly.

When expressed in terms of the dual 1-form, the Beltrami
equation~\eqref{Beltrami} takes the form
\[
d_{\RR^3}\be=f *_{\RR^3}\!\be\,.
\]
Reading off the coefficients from~\eqref{dal}-\eqref{*al} and using
the equation~\eqref{ft2}, the Beltrami equation in the coordinates~$(t,\xi)$
amounts to the following system:
\begin{subequations}
\begin{align}
\pd_t\be_1&=(1+t)\chi|g|^{1/2}\,  g^{2i}\be_i\,,\label{dim2a}\\
\pd_t\be_2&=-(1+t)\chi|g|^{1/2}\,  g^{1i}\be_i\,,\label{dim2b}\\
0&=\pd_1\be_2-\pd_2\be_1\,.\label{constraint2}
\end{align}
\end{subequations}
To analyze this system, we begin by making use of Eq.~\eqref{al} to
consider $\be$ as a time-dependent 1-form on the surface $\Si$, which
maps each ``time'' $t$ to a 1-form in two dimensions
$\be(t)$. Eqs.~\eqref{dim2a}-\eqref{dim2b} show that the evolution in
time of this 1-form is defined by a time-dependent tensor field $T(t)$
on $\Si$ as
\begin{equation}\label{pdtal}
\pd_t\be=T(t)\, \be\,.
\end{equation}
In fact, $T(t)$ can be written in terms of the Hodge operator $*_t$
associated with the time-dependent metric 
\[
g_{ij}(t,\xi)\, d\xi_i\, d\xi_j
\]
on $\Si$ as
\begin{equation}\label{T}
T(t)\be=-(1+t)\chi(t,\xi)\, *_t\!\be
\end{equation}
and its components are
\begin{equation}\label{Tcomp}
(T^j_i)=(1+t)\chi|g|^{1/2}\left( \begin{array}{cc}
g^{12} & g^{22} \\
-g^{11} & -g^{12} \end{array} \right)
\end{equation}

On the contrary, Eq.~\eqref{constraint2} does not describe an evolution,
but impose the stationary constraint that $\be(t)$ must be closed (as
a 1-form on $\Si$) for all times. Denoting by $d$ the exterior
differential on the surface, this reads as
\begin{equation}\label{dal=0}
d\be=0\,.
\end{equation}

\section{Proof of the theorems}
\label{S.proofs}

To derive a useful necessary condition for $\be$ to be a solution of
the system~\eqref{pdtal}-~\eqref{dal=0}, which is equivalent to the
Beltrami equation~\eqref{Beltrami}, let us begin by defining the
family of time-dependent tensor fields $T_n(t)$ recursively as
\begin{align*}
T_1&:=T\,,\\
T_{n+1}&:= \pd_t T_n+T_nT\,.
\end{align*}
It is not hard to see that $T_n(t)$ depends on $n$ derivatives of
$f$ in a non-local manner, the non-locality being due to the definition
of the coordinate system, and that $T_n(0)$ is a (local, nonlinear)
function of the first $n$ derivatives of $f$. These tensor fields can be used to describe the constraints of the
system due to the following

\begin{proposition}\label{P.constraint}
If the function $f$ is of class $C^{k,\al}$ and $n\leq k-1$, the time-dependent $1$-form
$\be$ must satisfy the constraint
\[
d(T_n \be)=0
\]
at all times.
\end{proposition}

\begin{proof}
  As the constraint equation~\eqref{dal=0} holds for all times, it
  trivially implies that the time derivatives of the 1-form $\be$ must
  satisfy the constraint
\[
d (\pd_t^n\be)=0\,,
\]
for any $n$. An easy induction argument using the evolution
equation~\eqref{pdtal} shows that
\[
\pd_t^n\be=T_n\be\,,
\]
so the proposition follows.
\end{proof}

By exploiting the previous constraints with $n=0,1$, we are now ready
to prove that there are no nontrivial Beltrami fields whenever $f$ has
a level set diffeomorphic to a sphere:

\begin{proof}[Proof of Theorem~\ref{T.spheres}]
Since we are assuming that the surface $\Si=f^{-1}(1)$ is a sphere, Eq.~\eqref{ft2} and the fact that the gradient of $f$
does not vanish on $\Si$ imply that $\Si_{t_0}:=f^{-1}(1+t_0)$ is also diffeomorphic to
$\SS^2$ for small enough $t_0$. Hence the constraint
equation~\eqref{dal=0} implies that there is a scalar function $\psi(t,\xi)$
such that
\[
\be=d\psi\,.
\]
By the construction of $\be$ and the regularity of $f$, the 1-form $\be$ is of
class $C^{1,\al}$, so $\psi(t,\xi)$ is a $C^{2,\al}$ function of
$\xi$. 

Taking into account the form of  the tensor field $T$ (cf.\
Eq.~\eqref{T}), we then find that
\begin{align*}
d(T\be)&=-(1+t)d(\chi\, *_t\! d\psi)\\
&=(1+t)\chi\,\Big[-d\! *_t\!d\psi-d(\log \chi)\wedge *_td\psi\Big]\\
&=-(1+t)\chi\,\Big[\De_t\psi+\langle \nabla_t(\log \chi),\nabla_t\psi\rangle_t\Big]\, \mathrm{Vol}_t\,,
\end{align*}
where the subscripts denote that the Laplacian, gradient and volume
form are computed on the sphere $\Si_t$ using the induced metric,
which has components $g_{ij}(t,\xi)$. Since $\chi$ is nonzero
for small enough $t$, Proposition~\ref{P.constraint} then ensures that
the equation
\[
\De_t\psi+\langle \nabla_t(\log \chi),\nabla_t\psi\rangle_t=0
\]
holds. As this equation satisfies the maximum principle in the closed
surface $\Si_t$, it follows that $\psi(t,\xi)$ is a constant that
depends only on $t$. Thus $\be\equiv0$ in a neighborhood of
$\Si$, and therefore everywhere by unique continuation.
\end{proof}

The proof of Theorem~\ref{T.main} also makes crucial use of
Proposition~\ref{P.constraint} to show that the function $f$ must
satisfy some differential constraint:

\begin{proof}[Proof of Theorem~\ref{T.main}]
Let us begin by recording the following formula for $d(T_n\be)$
in local coordinates:
\begin{multline}\label{dTnal}
d (T_n\be)=\Big[ \big(\pd_1 (T_n)^i_2-\pd_2(T_n)^i_1\big) \,\be_i+ \big(
(T_n)^2_2-(T_n)^1_1\big) \, \pd_2\be_1\\
+
(T_n)^1_2\,\pd_1\be_1- (T_n)^2_1\pd_2\be_2\Big]\, d\xi_1\wedge d\xi_2\,.
\end{multline}
Here we have used that $\pd_1\be_2=\pd_2\be_1$ by the constraint
equation~\eqref{constraint2}.

By Eq.~\eqref{Tcomp} and the fact that
$g_{ij}$ is  a metric,
\[
T^2_1=(1+t)\chi|g|^{1/2} g^{22}
\]
is strictly positive. Since $d(T\be)=0$ by Proposition~\ref{P.constraint}, we
can therefore isolate $\pd_2\be_2$ in this equation, finding that
\begin{equation}\label{pd2be2}
\pd_2\be_2=\frac1{T^2_1}\Big[ \big(\pd_1 T^i_2-\pd_2T^i_1\big) \,\be_i
+
T^1_2\,\pd_1\be_1+ (T^2_2-T^1_1)\,\pd_2\be_1\Big]\,.
\end{equation}

To simplify the notation, let us consider the $4$-component vectors
\[
\Ga:=(\be_1,\be_2,\pd_1\be_1,\pd_2\be_1)
\]
and
\begin{equation}\label{cTn}
\cT_n :=\left( \begin{array}{c}
\pd_1 (T_n)^1_2-\pd_2(T_n)^1_1
-\frac{(T_n)^2_1}{T^2_1}\big(\pd_1 T^1_2-\pd_2T^1_1\big) \\[1mm]
\pd_1 (T_n)^2_2-\pd_2(T_n)^2_1
-\frac{(T_n)^2_1}{T^2_1}\big(\pd_1 T^2_2-\pd_2T^2_1\big)\\[1mm]
(T_n)^1_2
-\frac{(T_n)^2_1 T^1_2}{T^2_1}\\[1mm]
(T_n)^2_2-(T_n)^1_1-\frac{(T_n)^2_1}{T^2_1}\big(T^2_2-T^1_1\big)
  \end{array}\right)
\end{equation}
Using Eqs.~\eqref{constraint2} and~\eqref{pd2be2} in~\eqref{dTnal}, one can then write
\[
d(T_n\be)=(\Ga\cdot \cT_n)\, d\xi_1\wedge d\xi_2\,,
\]
where the dot has the obvious meaning. Hence the contraint
$d(T_n\be)=0$ granted by Proposition~\ref{P.constraint} takes the form
\begin{equation}\label{cTnGa}
\cT_n\cdot\Ga=0\,,
\end{equation}
the condition being a priori nontrivial for all $n\geq2$.  In
particular, if the Beltrami equation has a nonzero solution, the
matrix $(\cT_2,\cT_3,\cT_4,\cT_5)$ cannot be of maximal rank, that is,
\begin{equation}\label{det}
\det(\cT_2,\cT_3,\cT_4,\cT_5)=0\,.
\end{equation}
Due to the definition of the tensor fields $T_n$ and~$\cT_n$, this equation involves derivatives of $f$ of order at
most~6.

Eq.~\eqref{det} is almost the differential constraint $P[f]=0$ whose
existence was claimed in the statement of the theorem. The only subtle
point is that, as we discussed when we defined the tensor fields
$T_n$, for $t\neq0$ Eq.~\eqref{det} is not a local function of~$f$ because we have used the implicit function theorem and the
flow of the vector field $X$ to construct the local coordinate
system~$(t,\xi)$. However the differential
constraint can be defined on the initial surface $\Si$ as
\begin{equation}\label{detPf}
P[f]|_\Si:=\det(\cT_2,\cT_3,\cT_4,\cT_5)\big|_{t=0}\,,
\end{equation}
and this is indeed a local function of $f$ and its derivatives up to
sixth order.  Since the initial surface $\Si$ is arbitrary and the
dependence on the surface is smooth, by carrying out the same
construction with $\Si$ replaced by $\Si_{t_0}$, for all $t_0$ in an
interval around $0$, this defines $P[f]$ in a neighborhood of $\Si$,
thereby completing the proof of the theorem.
\end{proof}

To show that the nonlinear differential operator $P$ is nontrivial,
we will approximately compute its action on a couple of concrete
functions $f$. As a byproduct, we will provide a few simple, explicit examples of
functions for which there are no nontrivial solutions to the Beltrami
equation~\eqref{Beltrami} and illustrate how one can evaluate~$P[f]$ in practice. More sophisticated examples can be obviously
obtained using the same procedure. Of course, one can easily check
that $P[\la]\equiv 0$ for any constant $\la$.

\begin{proposition}\label{P.example1}
Suppose that the vector field $u$ satisfies the Beltrami equation~\eqref{Beltrami} with
\[
f(x):=1+a x_1+bx_1^3+x_3
\]
in a neighborhood of the origin. Then $u\equiv0$ if $b\neq0$.
\end{proposition}

\begin{proof}
Using the same notation as in the proof of Theorem~\ref{T.main},
we can parametrize $\Si:=f^{-1}(1)$ as the graph
\[
\Si=\big\{(\xi,h(\xi)\big\}\,,
\]
with
\[
h(\xi):=-a\xi_1-b\xi_1^3\,.
\]
After a lengthy but straightforward computation starting from Eq.~\eqref{xtxi}, one can compute the remaining objects that appear in the
proof of Theorem~\ref{T.main} as a power series in the coordinate
$t$. In particular, one finds that the determinant~\eqref{detPf} takes
the form
\[
P[f]|_\Si =\sum_{j=0}^4 c_j\xi_1^j+O(\xi_1^5)\,,
\]
where the coefficients $c_j$ depend on $a$ and $b$ as:
\begin{align*}
c_0&:=-\frac{5184 a^2 b^4 \left(15 a^4+14 a^2+36 a
   b-1\right)}{\left(a^2+1\right)^{14}}\,,\\
c_1&:=-\frac{20736 a^2 b^4 \left(8 a^3-63 a^2 b+8 a-9
   b\right)}{\left(a^2+1\right)^{14}}\,,\\
c_2&:=\frac{31104 a b^5 \left(169 a^6+97 a^4+468 a^3 b-73
   a^2-36 a b-1\right)}{\left(a^2+1\right)^{15}}\,,\\
c_3&:=\frac{124416 a^2 b^5 \left(84 a^4-771 a^3 b+68 a^2-15
   a b-16\right)}{\left(a^2+1\right)^{15}}\,,\\
c_4&:=46656 b^6+ ab\, G \,.
\end{align*}
Here $G$ is a complicated smooth rational function of $a$ and $b$
that can be computed explicitly.

The point now is that the only solution to the system of algebraic equations
\[
c_j=0\qquad \text{for }0\leq j\leq 4
\]
is $b=0$. In order to see this, a simple
computation shows that imposing $c_j=0$ for $0\leq j\leq 3$ implies
that $ab=0$, while $c_4=0$ for $b=0$ but not for $a=0$. The
proposition then follows from
Theorem~\ref{T.main}.
\end{proof}

\begin{remark}
For $b=0$, the function $f$ is affine and the Beltrami
equation~\eqref{Beltrami} does admit a nontrivial
solution, which is in fact defined in the whole space. Specifically,
if $u_0$ is any vector in $\RR^3$ orthogonal to $e:=(a,0,1)$,
\[
u:=u_0\, \cos\frac{(1+ax_1+x_3)^2}{2|e|}+ \frac{u_0\times e}{|e|}\, \sin\frac{(1+ax_1+x_3)^2}{2|e|}
\]
is such a solution. Hence the set of obstructions on $f$ that we
get from the operator $P$ is optimal for this family of functions.
\end{remark}

\begin{proposition}\label{P.example2}
Suppose that the vector field $u$ satisfies the Beltrami equation~\eqref{Beltrami} with
\[
f(x):=1+x_1^2+ax_2^2+x_3
\]
in a neighborhood of the origin. Then $u\equiv0$ if $a\neq1$.
\end{proposition}

\begin{proof}
As before,
we parametrize $\Si:=f^{-1}(1)$ as the graph
\[
\Si=\big\{(\xi,h(\xi)\big\}\,,
\]
with
\[
h(\xi):=-\xi_1^2-a\xi_2^2\,.
\]
Arguing as in Proposition~\ref{P.example1} one finds that the
determinant~\eqref{detPf} is of the form
\begin{multline*}
P[f]|_\Si =1024 (a-1)^2\Big[ (33 + 128 a + 312 a^2 + 224 a^3 + 768 a^4 - 256 a^5)\xi_1^2 \\
-16 a^2 (3 + 11 a + 66 a^2 - 88 a^3 + 8 a^4)\xi_1\xi_2 \\
+ a^4(-39 - 24 a + 760 a^2 + 640 a^3 - 128
a^4)\xi_2^2\Big]+O(|\xi|^3)\,.
\end{multline*}
It can be easily checked that the quadratic part of this function vanishes if
and only if $a=1$, so the proposition follows from
Theorem~\ref{T.main}.
\end{proof}


\section{Final remarks}
\label{S.remarks}

Let us conclude with a few comments regarding the existence of
Beltrami flows, in view of the results we have established in this paper.

\subsection{Compressible Euler flows}

In Ref.~\cite{Yudovich}, considerable attention is paid to the bearing
of compressible Beltrami fields on the helical flow paradox. Using
Theorem~\ref{T.main} and the results that we proved in~\cite{Annals}
we can now show that compressible Beltrami fields have totally
different existence properties than the incompressible ones, since
whenever the function $f$ does not change sign one has many associated solutions.

More precisely, we have the following theorem. We recall that a
compressible Beltrami field is not a solution of the Euler equation
unless the barotropic condition is satisfied, and that it is natural
to assume that $f$ is positive because it plays the role of the fluid
density.

\begin{theorem}\label{T.compressible}
Let $f$ be a positive real-analytic function in $\RR^3$. Then there
are nontrivial solutions to the equation
\begin{equation}\label{incompr}
\curl u=f u
\end{equation}
defined in the whole space $\RR^3$.
\end{theorem}

\begin{proof}
We proved in~\cite[Example~8.2]{Annals} that if $\tilde g$ is an
analytic (possibly incomplete) Riemannian metric in $\RR^3$, there is
a vector field $v$, not identically zero, which satisfies the
equation
\[
\curl_{\tilde g}v=v
\]
in $\RR^3$. Here $\curl_{\tilde g}$ denotes the curl operator
associated with the metric $\tilde g$. 

Let us now choose $\tilde g$ as the conformally flat metric
\[
\tilde g:=f^2 g_0\,,
\]
where $g_0$ denotes the Euclidean metric.  If we set $u:=f^2 v$, a
straightforward computation shows that the Euclidean curl of $u$ is
given by
\[
\curl u= fu\,,
\]
thus completing the proof of the theorem.
\end{proof}

\begin{remark}
In particular, a straightforward consequence
of~\cite[Example~8.2]{Annals} and of the proof of the theorem is the
following: if $f$ is a positive analytic function and $L$ is any locally finite link in $\RR^3$, one can
transform it using a smooth diffeomorphism $\Phi$ of $\RR^3$ so that
$\Phi(L)$ is a set of vortex lines of the vector field $u$, which
satisfies Eq.~\eqref{incompr} in $\RR^3$. Furthermore, $\Phi$ can be
chosen close to the identity in any $C^k$ norm. Hence, there is much
freedom in the choice of the nontrivial solution~$u$.
\end{remark}

\subsection{Strong Beltrami fields}

When the function $f$ equals some nonzero constant $\la$, it is well
known that the Beltrami equation~\eqref{Beltrami} has an infinite
number of solutions. In particular, if $c_{lm}$ is a set of constant
vectors in $\RR^3$ for which the sum
\[
u:=\curl (\curl + \la)\sum_{l=0}^\infty\sum_{m=-l}^l c_{lm}\, j_l(\la r)\, Y_{lm}(\te,\vp)
\]
converges in a suitable sense, $u$ is a Beltrami field with constant
$\la$.  Here we are using spherical coordinates, $j_l$ is the
spherical Bessel function and $Y_{lm}$ are the spherical harmonics.

However, the results we have proved in this paper also have an
implication about strong Beltrami fields. In fact, it can be readily
checked that the proofs of the main results remain valid under the
assumption that $u$ is a strong Beltrami field and $f$~is a first
integral of $u$. Hence we get for free the following

\begin{theorem}\label{T.strong}
Assume that $u$ is a strong Beltrami field in a domain
$U\subseteq\RR^3$. Then it cannot have a first integral of class
$C^{2,\al}(U)$ with a regular level set diffeomorphic to
$\SS^2$. Furthermore, a (nonconstant) function $f\in C^{6,\al}(U)$ cannot be a first integral of $u$ unless it satisfies the
equation $P[f]=0$, where $P$ is a nonlinear differential operator of
sixth order that
does not depend on the particular Beltrami field $u$ and which can be
computed explicitly.
\end{theorem}

Notice that the assertion that a first integral of a Beltrami field
cannot have a level set diffeomorphic to the sphere is
reminiscent of (and somehow complementary to) Arnold's structure
theorem~\cite{Ar66} for steady solutions of the Euler equation with
nonconstant Bernoulli function (that is, for solutions where $u$ and $\curl u$
are not collinear). In this case, the compact level
sets of the Bernoulli function must be tori.

\subsection{Further differential constraints}

We saw in the proof of Theorem~\ref{T.main} that the differential
operator $P$ that yields the constraints for the function~$f$ is given
in terms of the 4-component vectors $\cT_n$ defined in~\eqref{cTn}
via
\[
P[f]|_\Si:=\det(\cT_2,\cT_3,\cT_4,\cT_5)\big|_{t=0}\,.
\]

An important observation is that the proof of Theorem~\ref{T.main}
actually gives more information that the statement of the
theorem. Actually, it is a straightforward
consequence of Eq.~\eqref{cTnGa} that $f$ must also satisfy the
differential equation
$P_{ijkl}[f]=0$, where we set
\[
P_{ijkl}[f]|_\Si:=\det(\cT_i,\cT_j,\cT_k,\cT_l)\big|_{t=0}
\]
for any integers 
\[
l>k>j>i\geq 2\,,
\]
provided that $f$ is smooth enough (e.g., of class $C^{l+1,\al}$). Therefore, one
would expect to have a hierarchy of differential constraints on a
smooth $f$ to admit nontrivial solutions. Notice that proving the
independence of the resulting system of constraints should be a
delicate problem due to the complexity of the expressions for
$P_{ijkl}$.

\subsection{Steady Euler flows in two dimensions}

Let us consider a vector field  $v=\nabla^{\perp}\psi$, with $\psi$ a scalar function. It is standard that $v$ is a steady solution of the
incompressible Euler equation in two dimensions if the stream
function~$\psi$ satisfies the equation
\[
\De \psi=F( \psi)
\]
for some function~$F$. It is well known that this equation always admits a nontrivial solution $\psi$ for
smooth enough~$F$ in any ball that is sufficiently small.

Therefore, it stems that there are no local obstructions to the
existence of steady Euler flows in two dimensions for any smooth
function $F$. This is in sharp contrast with the Beltrami solutions of
the steady Euler equation in $\RR^3$ (cf.\ Theorem~\ref{T.main}). It
is worth mentioning that an important recent contribution to the study of the
geometry of the space of steady solutions in two-dimensional domains
is~\cite{Sverak}, where Arnold's approach to Euler flows using
volume-preserving diffeomorphisms is revisited.

\section*{Acknowledgments}

The authors are indebted to Boris Khesin for valuable comments on the
manuscript. A.E.~and D.P.S.~are respectively supported by the Spanish MINECO through the
Ram\'on y Cajal program and by the ERC
grant~335079. This work is supported in part by the
grants~FIS2011-22566 (A.E.), MTM2010-21186-C02-01 (D.P.S.) and
SEV-2011-0087 (A.E.\ and D.P.S.).

\bibliographystyle{amsplain}

\end{document}